\documentclass[11pt]{amsart}
\usepackage{amssymb,amsxtra}%,amscd,amsthm,amsmath,amsfonts}

\newcommand{\norm}[1]{\mbox{$\|#1\|$}}
\newcommand{\x}{\times}
\newcommand{\cs}{\mbox{$C^{*}$-algebra}}
\newcommand{\css}{\mbox{$C^{*}$-algebras}}

\newcommand{\C}{\mathbb{C}}

\newcommand{\R}{\mathbb{R}}

\newcommand{\ov}[1]{\mbox{$\overline{#1}$}}

\newcommand{\al}{\mbox{$\alpha$}}
\newcommand{\eps}{\mbox{$\epsilon$}}
\newcommand{\bt}{\mbox{$\beta$}}
\newcommand{\ga}{\mbox{$\gamma$}}
\newcommand{\Ga}{\mbox{$\Gamma$}}
\newcommand{\de}{\mbox{$\delta$}}
\newcommand{\De}{\mbox{$\Delta$}}

\newcommand{\la}{\mbox{$\lambda$}}

\newcommand{\si}{\mbox{$\sigma$}}

\newcommand{\om}{\mbox{$\omega$}}

\newcommand{\mfm}{\mathfrak{m}}
\newcommand{\mfn}{\mathfrak{n}}

\newcommand{\mfH}{\mathfrak{H}}

\newcommand{\bc}{\begin{center}}

\newcommand{\ec}{\end{center}}
\newcommand{\be}{\begin{enumerate}}
\newcommand{\ee}{\end{enumerate}}
\newcommand{\beqn}{\begin{eqnarray}}
\newcommand{\eeqn}{\end{eqnarray}}
\newcommand{\beqns}{\begin{eqnarray*}}
\newcommand{\eeqns}{\end{eqnarray*}}
\newcommand{\bq}{\begin{quote}}
\newcommand{\eq}{\end{quote}}

\newcommand{\bi}{\begin{itemize}}
\newcommand{\ei}{\end{itemize}}
\newcommand{\bd}{\begin{description}}
\newcommand{\ed}{\end{description}}

\theoremstyle{plain}
\newtheorem{theorem}{Theorem}%[section]

\newtheorem{proposition}[theorem]{Proposition}
\newtheorem{corollary}[theorem]{Corollary}

\theoremstyle{definition}
\newtheorem{definition}[theorem]{Definition}

\theoremstyle{remark}

\begin{document}
\title{Group amenability properties for von Neumann algebras}

\author{Anthony T. Lau}
\address{Department of Mathematical and Statistical Sciences, 
University of Alberta, Edmonton, AB T6G 2G1}
\email{tlau@math.ualberta.ca}
\thanks{The first author is supported by the NSERC grant A-7679.}
\author{Alan L. T. Paterson}
\address{Department of Mathematics, University of Mississippi, University,
Mississippi 38677}
\email{mmap@olemiss.edu}

\date{January, 2005}

\begin{abstract}
In his study of amenable unitary representations, M. E. B. Bekka asked if there is an analogue for such representations of the remarkable fixed-point property for amenable groups.  In this paper, we prove such a fixed-point theorem in the more general context of a $G$-amenable von Neumann algebra $M$, where $G$ is a locally compact group acting on $M$.   
The F\o lner conditions of Connes and Bekka are extended to the case where $M$ is semifinite and admits a faithful, semifinite, normal trace which is invariant under the action of $G$.
\end{abstract}

\maketitle

\section{Introduction}%1

In \cite{Bekka}, Bekka introduced the concept of an {\em amenable} representation of a locally compact group $G$ on a Hilbert space.  A continuous unitary representation $\pi$ of $G$ on a Hilbert space $\mfH_{\pi}$ is defined to be amenable if there exists a state $p$ on $B(\mfH_{\pi})$ that is invariant in the sense that 
\begin{equation}  \label{eq:gamen}
    \pi(x)p\pi(x)^{-1}=p
\end{equation}         
for all $x\in G$.  Bekka developed a remarkable theory of amenable representations, paralleling the classical theory of amenable groups, and raised the question concerning what should be the fixed point property for such representations.  In this paper we obtain a solution to this question, and develop, more generally, a theory of $G$-amenable von Neumann algebras $M$ for group actions on $M$.  We also obtain a F\o lner condition in the case where $M$ is semifinite, generalizing the case $M=B(\mfH)$ considered by Connes (\cite{Connes}) and Bekka.  This generalization substantially widens the applicability of this condition.   We now discuss in a little more detail the contents of the paper.

The second section describes some basic facts about Banach $G-A$-modules where $A$ is a Banach algebra. These will be applied in the case where $A$ is the predual of a von Neumann algebra $M$.  The third section discusses the notion of $G$-amenability for a von Neumann algebra $M$: $M$ is called {\em $G$-amenable} for an automorphic action $x\to \al_{x}$ of $G$ on $M$ if there exists a state $p$ on $M$ such that $(\al_{x^{-1}})^{*}(p)=p$.  (This notion of $G$-amenability has also been considered by R. Stokke (\cite{Stokke}).)  In the situation of (\ref{eq:gamen}), $\pi$ being an amenable representation is the same as $B(\mfH_{\pi})$ being $G$-amenable under the action: $\al_{x}(T)=\pi(x)T\pi(x)^{-1}$.
We also discuss some fairly familiar amenability properties of $G$-actions on $M$ paralleling those of classical amenability theory for $L^{\infty}(G)$, and introduce a notion of inner actions for
$A=M_{*}$.  The section concludes with a discussion of {\em hyperstates}: these were introduced for von Neumann algebras by Connes.  

In the fourth section, we discuss our fixed-point property characterizing $G$-amenability.  Examples show that there is no hope of there being a fixed-point theorem involving a compact convex set as in normal amenability theory.   We are motivated by \cite{LauP}, where a fixed-point property characterizing inner amenability of a locally compact group $G$ is proved.  This has the following form: {\em $G$ is inner amenable (i.e. $G$ admits a mean on $L^{\infty}(G)$ invariant under conjugation) if and only if, whenever $X$ is a Banach $G$-module, there exists a net of multiplication operators on $X^{**}$ coming from $P(G)$ that converges weak$^{*}$ to an operator $T\in B(X^{**})$ which is invariant under conjugation by elements of $G$.}  Here, the weak$^{*}$-topology on $B(X^{**})$ is that obtained by identifying $B(X^{**})$ with the Banach dual space $(X^{**}\hat{\otimes} X^{*})^{*}$ in the canonical way.  We show in the present paper that a similar fixed-point property characterizes $G$-amenability for a wide class of von Neumann algebras $M$.  The predual of $M$ has to be a Banach algebra $A$ for which the action of $G$ on $M$ induces a Banach algebra action.  Examples of such $M$'s are $B(\mfH)$ (where the action is determined by an amenable representation $\pi$ of $G$ and $A$ is the algebra of trace class operators) and $L^{\infty}(G)$ (under the action induced by conjugation and where $A=L^{1}(G)$).  In our fixed-point property, the group algebra $L^{1}(G)$ of \cite{LauP} is replaced by the covariant algebra $L^{1}(G,A^{1})$, where $A^{1}$ is the Banach algebra $A$ with identity adjoined.  In general, $A$ does not have a bounded approximate identity.  We show that when $A$ {\em does} have a bounded approximate identity and the $G$-action is inner for $A$, then the fixed-point property is required only for essential left Banach $A$-modules.  Further, in this case, we can replace 
$L^{1}(G,A^{1})$ by $L^{1}(G,A)$.

In the fifth section, we describe a number of examples illustrating the theory.  In one of these examples, $M$ is a Hopf-von Neumann algebra: in that situation there is a natural Banach algebra structure on $A$ to which the theory applies.  We discuss briefly in the sixth section the fixed-point property for Hopf-von Neumann algebras in general.  The last section proves a F\o lner condition that characterizes $G$-amenability for semifinite $M$.  The proof follows that of Bekka's, the 
F\o lner sets of classical amenability being replaced by (finite) projections $P$ in $M$ for which $0<\tau(P)<\infty$, where $\tau$ is an invariant, faithful, semifinite, normal trace on $M$.   When $M=B(\mfH)$ and the action of $G$ on $M$ is given by a representation of $G$ in $M$, then our F\o lner condition reduces to that of \cite{Bekka}.  The classical 
F\o lner condition is also a special case of our theorem.  In addition, Connes (\cite{Connes}) proved a result in which, under certain conditions, the tracial state on a $II_{1}$-factor $N$ is approximated by states determined by the finite dimensional projections arising in the F\o lner condition.  A version of this result, with $N$ replaced by a $\cs$ associated with a group representation in $M$, is given to conclude the paper.

The authors are grateful to Paul Jolissaint bringing to their attention (after the present paper appeared) his earlier paper \cite{Jol} which gives a different fixed point property that also answers (among other things) Bekka's question.   More precisely, and in our notation, suppose that $X$ is an isometric Banach $G$-module with action $\bt$ and $K$ is a 
$G$-invariant weak$^{*}$-compact, convex subset of the unit ball of $X^{*}$.  Let $C(M,K)$ be the set of bounded linear maps $T$ from $A=M_{*}$ to $X^{*}$ such that $T(S_{*}(M))\subset K$.   There is a natural $G$-action on $C(M,K)$ given by:
$g.T=(\bt_{g})^{*}\circ T\circ ((\al_{g})_{*})^{-1}$.   Then the triple $(M,G,\al)$ is said to have the 
{\em conditional fixed point property} if for all such $X, K$, the existence of a $G$-invariant element of $C(M,K)$
implies the existence of a $G$-fixed point in $K$.  It is shown in \cite{Jol} that $(M,G,\al)$ has the conditional fixed point property if and only if there exists an $\al$-invariant state on $M$.  This fixed point property does not require 
$A$ to be a Banach algebra as does our Theorem~\ref{th:fpt}.  On the other hand, while the latter theorem requires 
a $G-A$-fixed point for the {\em fixed} algebra $L^{1}(G,A^{1})$, the $C(M,K)$'s of the conditional fixed point property vary from $K$ to $K$.   Jolissaint also characterizes his fixed point property in terms of a conditional expectation onto $W^{*}(G)$.  The two fixed point theorems seem very different in character, and it would be interesting to explore the relation between them.

%*****************
%*****************
%*****************

\section{Banach $G-A$ modules}%2

We first recall some standard terminology.  Let $G$ be a locally compact group with left Haar measure $\la$.  The family of compact subsets of $G$ is denoted by $\mathcal{C}(G)$, and $C_{c}(G)$ is the space of continuous, complex-valued functions on $G$ with compact support.  The set of probability measures in $L^{1}(G)$ is denoted by $P(G)$. 
The universal $C^{*}$-algebra of $G$ is denoted by $C^{*}(G)$ and the reduced $C^{*}$-algebra of $G$ by $C^{*}_{\rho}(G)$.

The Banach algebra of bounded linear operators on a Banach space $X$ is denoted by $B(X)$.  The Banach space $X$ is called a {\em left Banach $G$-module} if for each $x\in G$, we are given a map
$\xi\to x\xi$ in $B(X)$, with $\norm{x\xi}\leq \norm{\xi}$ for all $x\in G, \xi\in X$, and such that $x(y\xi)=(xy)\xi$ for all 
$x,y\in G$, and for fixed $\xi\in X$, the map $x\to x\xi$ is norm continuous.  The action of $G$ on such an $X$ integrates up to give a representation of $L^{1}(G)$ on $X$:
\[                       f\xi=\int f(x)x\xi\,d\la(x).                      \]
Now let $A$ be a Banach algebra.  A Banach space $X$ is called a {\em left Banach $A$-module} if it is a left $A$-module such that $\norm{a\xi}\leq \norm{a}\norm{\xi}$ for all $a\in A, \xi\in X$.  The left Banach $A$-module $X$ is called {\em essential} if for all $\xi\in X$, there exist $a\in A$ and $\eta\in X$ such that $a\eta=\xi$.  Right and two-sided Banach $G$-modules (resp. Banach $A$-modules) are defined in the obvious way.

We next recall some ``Arens product'' terminology.  Let $A$ be a Banach algebra and 
$X$ be a left Banach $A$-module.  Then for $a\in A$, $f\in X^{*}$, the functional
$fa\in X^{*}$ is defined by: $fa(\xi)=f(a\xi)$.  Now for $m\in X^{**}$, we define $mf\in A^{*}$ by: $mf(a)=m(fa)$.  Lastly for $n\in A^{**}$, we define $nm\in X^{**}$ by: $nm(f)=n(mf)$.  Let $T_{n}\in B(X^{**})$ be given by: $T_{n}m=nm$.
The above applies to $X=A$ with the $A$ action that of left multiplication; in that case $(n,m)\to nm$ is the first (left) Arens product on $A^{**}$.  With this product, $A^{**}$ is a Banach algebra.   The next proposition is well-known (e.g. \cite[p.527]{Palmer}).

\begin{proposition}           \label{prop:bna}
Let $B$ be a Banach algebra with a bounded approximate identity $\{e_{\de}\}$ and let $B^{**}$ be the second dual Banach algebra under the Arens multiplication. Then there exists $e\in B^{**}$ such that $me=m$ for all $m\in B^{**}$.  Further, $e$ can be taken to be any weak$^{*}$-cluster point of $\{\widehat{e_{\de}}\}$.
\end{proposition}

Now assume that there is an action of $G$ on $A$.  So we are given a bounded homomorphism $x\to \bt_{x}$ where each $\bt_{x}$ is an isomorphism of $A$ with $\norm{\bt_{x}}\leq 1$ and the map $x\to \bt_{x}(a)$ is norm continuous from $G$ into $A$ for each $a\in A$.  The map $x\to \bt_{x}$ dualizes to $A^{*}$ and $A^{**}$ in the natural way.  So for $x\in G, f\in A^{*}, a\in A$,
$\bt_{x}^{*}(f)\in A^{*}$ is given by: $\bt_{x}^{*}f(a)=f(\bt_{x}a)$, and for $m\in A^{**}$, $\bt_{x}^{**}m(f)=m(\bt_{x}^{*}f)$.  One readily shows that 
$(\bt_{x}^{*}f)a=\bt_{x}^{*}(f\bt_{x}a)$ and
$m\bt_{x}^{*}f=\bt_{x}^{*}[(\bt_{x}^{**}m)f]$ and hence that 
$\bt_{x}^{**}(nm)=\bt_{x}^{**}(n)\bt_{x}^{**}(m)$ for all $n,m\in A^{**}$.  

Next suppose that $X$ is a left Banach $G$-module that is also a left Banach $A$-module.
Assume further that $\bt_{x}(a)\xi=xax^{-1}\xi$ for all $x\in G, a\in A$ and $\xi\in X$.
Then we say that $X$ is a {\em left Banach $G-A$-module}.  (Such a module is the Banach space analogue of a covariant representation of a $\cs$ on a Hilbert space.)
Note that in this definition, $xax^{-1}$ usually has no meaning as an element of $A$ since $A$ need not be a $G$-module.  However, $xax^{-1}\xi=x(a(x^{-1}\xi))$ does make sense, and is required to be $\bt_{x}(a)\xi$.  

Of special importance is the case where there is given a strictly continuous homomorphism $x\to \ga(x)$ from $G$ into the
invertible group of the multiplier algebra $M(A)$ of $A$ such that $\norm{\ga(x)}\leq 1$ for all $x\in G$.  Write $xa, ax$ in place of $\ga(x)a, a\ga(x)$ ($a\in A$).  Then there is a natural action $x\to \bt_{x}$ of $G$ on $A$ given by: $\bt_{x}(a)=xax^{-1}$.  In this case, we call the action $\bt$ of $G$ on $A$ {\em inner}.  Suppose now that $A$ has a bounded approximate identity $\{e_{\de}\}$, that the action of $G$ on $A$ is inner and that $X$ is an essential left Banach $A$-module.  Then $X$ is automatically a left Banach $G-A$-module, the left $G$-module structure on $X$ being given by: $x\xi=\lim (xe_{\de})\xi$. 

Here are two examples of inner actions of $G$ on some $A$ relevant to this paper.  First, let $A=L^{1}(G)$ under convolution and $(\ga(x)f)(t)=f(x^{-1}t), (f\ga(x))(t)=f(tx^{-1})\De(x^{-1})$ where $\De$ is the modular function of $G$.  In the notation of \cite[(0.7)]{Paterson}, 
$\ga(x)f=x*f=\de_{x}*f$, and  $f\ga(x)=f*x=f*\de_{x}$.  Then $\bt_{x}(f)(t)=x*f*x^{-1}(t)=f(x^{-1}tx)\De(x)$, and $\bt$ is inner on $L^{1}(G)$. For the second example, we take $\pi$ to be a unitary representation of $G$ on a Hilbert space $\mfH_{\pi}$, $A=TC(\mfH_{\pi})$, the Banach algebra of trace class operators on $\mfH_{\pi}$, and $\ga(x)=\pi(x)$, so that $\bt_{x}(a)=\pi(x)a\pi(x^{-1})$. In the first case, $A$ has a bounded approximate identity, but not (as we will see) in the second case.   

Now let $L^{1}(G,A)$ be the covariant algebra associated with an action $\bt$ of $G$ on a Banach algebra $A$.  The theory of $C^{*}$-dynamical systems, developed in \cite[7.6]{Pedersen}, applies in this Banach algebra context.  Below, except where special comments are made, the $C^{*}$-dynamical proofs apply equally to the Banach algebra case.  Let $C_{c}(G,A)$ be the normed algebra of continuous functions $f:G\to A$ with compact support, with norm and product defined by:
\[\int\norm{f(x)}\,d\la(x),\hspace{.2in} f*g(x)=\int f(y)\bt_{y}(g(y^{-1}x))\,d\la(y).   \]
The Banach algebra $L^{1}(G,A)$ is the completion of $C_{c}(G,A)$. 

Assume now that $A$ has a bounded approximate identity.  Then $L^{1}(G,A)$ has a bounded approximate identity in $C_{c}(G,A)$, an example of which can be taken to be $\{f_{\de}\otimes e_{\de}\}$ where $\{f_{\de}\}$ is a bounded approximate identity for $L^{1}(G)$ of the standard kind and $\{e_{\de}\}$ is a bounded approximate identity  in $A$.  For each $a\in A, \mu\in M(G)$, where $\mu$ has compact support, there are respectively left and right multipliers $L(a,\mu), R(a,\mu)$ of $L^{1}(G,A)$, where for $f\in L^{1}(G,A)$:
\begin{gather}
 L(a,\mu)f(t)  = a\int \bt_{y}(f(y^{-1}t))\,d\mu(y),  \label{eq:lamu} \\
R(a,\mu)f(t)  =  \int f(ty^{-1})\bt_{ty^{-1}}(a)\De(y)^{-1}\,d\mu(y). \label{eq:ramu}
\end{gather}
(These equalities are first defined for $f\in C_{c}(G,A)$ and then extended by continuity to $L^{1}(G,A)$.)
Further, the pair 
$(L(a,\mu), R(a,\mu))$ is a (two-sided) multiplier for $L^{1}(G,A)$.  Next, the elements of both
$G$ and $A$ act as multipliers on $L^{1}(G,A)$.  Indeed, for $a\in A, x\in G$, we take $af=L(a,\de_{e})f$,
$fa=R(a,\de_{e})f$, $xf=\lim_{\de}L(e_{\de},\de_{x})f$, and 
$fx= \lim_{\de}R(e_{\de},\de_{x})f$, these limits being taken in the norm topology.
For example, $(af)(t)=a(f(t))$ and the multiplier condition $(af)*g=a(f*g)$ holds since
\[ a(f*g)(t)=a\int f(y)\al_{t}(g(y^{-1}t))\,d\la(y)=((af)*g)(t).          \]
It is readily checked that $fa(t)=f(t)\bt_{t}(a)$ and 
\begin{equation}        \label{eq:yffy}
xf(t)=\bt_{x}(f(x^{-1}t))  \hspace{.2in} fx(t)=f(tx^{-1})\De(x^{-1}).
\end{equation}
Further, it is straightforward to check that 
\begin{equation}            \label{eq:btxax}
\bt_{x}(a)f=xax^{-1}f    
\end{equation}
for all $x\in G, a\in A, f\in L^{1}(G,A)$.  So $L^{1}(G,A)$ is a left Banach $G-A$-module.
More generally, if $X$ is an essential left Banach $L^{1}(G,A)$-module, then (e.g. \cite[pp. 15-16]{Lance}, \cite[p.256]{Pedersen}) $X$ is a left Banach $G-A$-module and the $L^{1}(G,A)$ action on $X$ is given by integration.  (See (\ref{eq:integ}) below.)

We note that $L^{1}(G,A)$ is an {\em essential} left Banach $A$-module.  Indeed, since $A$ has a bounded approximate identity, we have $A^{2}=A$, and so the set $W=\{af: a\in A, f\in L^{1}(G,A)\}$ contains $L^{1}(G)\otimes A$.  Since $W$ is a closed linear subspace of $L^{1}(G,A)$ (\cite[Theorem 32.22]{HR2}), it follows that $W=L^{1}(G,A)$, and so $L^{1}(G,A)$ is an essential left Banach $A$-module.

Now remove the assumption that $A$ has a bounded approximate identity and let $X$ be a left Banach $G-A$-module.   The actions of $G, A$ on such a module $X$ integrate up to give that $X$ is a left Banach $L^{1}(G,A)$-module: for $f\in L^{1}(G,A)$, $f\xi$ is given by:
\begin{equation}
 f\xi=\int f(x)x\xi\,d\la(x).     \label{eq:integ}
\end{equation}
In general, we cannot reverse this implication.        

%*****************
%*****************
%*****************

\section{$G$-amenability}%3

In this section we discuss the notion of $G$-amenability for a von Neumann algebra $M$ and some of its equivalent formulations.  Most of this is effectively in the papers \cite{Bekka,Stokke} but for the convenience of the reader, we will sketch proofs in some cases.  

So let $M$ be a von Neumann algebra $M$ with predual $M_{*}$.  Let $S(M)$ be the set of states on $M$ and $S_{*}(M)$ the set of normal states in $M_{*}$, i.e. $S_{*}(M)=S(M)\cap M_{*}$.  We note that $S_{*}(M)$ is weak$^{*}$-dense in $S(M)$.  Let $G$ be a locally compact group with an (automorphic) action on $M$ (\cite[p.238]{Takv2}).  This means that we are given a homomorphism $x\to \al_{x}$ from $G$ into $Aut(M)$ such that for each $m\in M$, the map $x\to \al_{x}(m)$ is weak$^{*}$-continuous.  Then $M$ is a right $G$-module under the map: $(x,m)\to m.x=\al_{x^{-1}}(m)$.  Dualizing gives that $M^{*}$ 
(resp. $M_{*}$) are left $G$-modules under the map: $(x,c)\to x.c=(\al_{x^{-1}})^{*}(c)$ 
(resp. $(\al_{x^{-1}})_{*}(c)$).  Further (\cite[p.239]{Takv2}) $M_{*}$ is  actually a left Banach $G$-module: the map $x\to x.a$ is norm continuous for each $a\in M_{*}$.   Note that this gives a positive action of $G$ on $M_{*}$ (in the sense of Stokke (\cite[p.3]{Stokke}), i.e. for all $x\in G$, $x.S_{*}(M)=S_{*}(M)$.  Also, $x.S(M)=S(M)$.
Of course, we obtain $M_{*}$ as a Banach $L^{1}(G)$-module by integrating: so, for example, for $a\in M_{*}, f\in L^{1}(G)$, $f.a=\int f(x)x.a\,d\la(x)$ (Bochner integral).   This dualizes to give that $M$ is a right Banach $L^{1}(G)$-module and $M^{*}$ is a left Banach $L^{1}(G)$-module.  $M$ is a right Banach $L^{1}(G)$-module by integrating weakly: so for $f\in L^{1}(G), m\in M, a\in M_{*}$, we have 
\[     \left<m.f,a\right>=\int_{G}f(x)\left<m.x,a\right>\,d\la(x).                             \]

We will sometimes require the action $\al$ to be {\em inner on $M$} in the following sense.    For this, $M$ is a $G$-bimodule with the property that for each $m\in M$, the maps $x\to xm, x\to mx$ are bounded and weak$^{*}$-continuous, and $\al_{x}(m)=xmx^{-1}$.   An example of this is given by (a) $M=L^{\infty}(G)$ with $\al_{x}\phi(t)=\phi(x^{-1}tx)$ ($t\in G$).    In this case, the $G$-module structure on $M$ is given by: $x\phi(t)=\phi(tx), \phi x(t)=\phi(xt)$.  We can also take the $G$-module structure to be given by (b) $x\to x\phi, x\to \phi$, in which case $\al_{x}(\phi)=x\phi$,  and (c) 
$x\to \phi, x\to \phi x$ in which case $\al_{x}(\phi)=\phi x^{-1}$.   An example of an inner action on $M$ of particular importance is that determined by a weak$^{*}$-continuous homomorphism $x\to \pi(x)$ into the unitary group of $M$.  We will call $\pi$ a {\em representation of $G$ in $M$}.   The $G$-module structure on $M$ is given by: $xm=\pi(x)m, mx=m\pi(x)$ and then $\al_{x}(m)=\pi(x)m\pi(x^{-1})$.   This is the context for Bekka's theory of amenable representations.  In the cases $M=L^{\infty}(G)$ above, the actions $\al$ are not in general given by representations of $G$ in $M$ but - as is to be expected at least for standard semifinite $M$ (\cite[Part 1, Chapter 6, Theorem 4]{D1}) - the $\al_{x}$'s are spatial.  In fact in (a), (b) and (c) above, $\al$ is determined by a representation of $G$ on $L^{2}(G)$.  Indeed, regard 
$L^{\infty}(G)\subset B(L^{2}(G))$ in the usual way; then in case (a), the action $\al$ is determined by the conjugation representation $\pi'$ of $G$ on $L^{2}(G)$.  So  $\al_{x}(\phi)=\pi'(x)\phi\pi'(x^{-1})$ where 
$\pi'(x)f(t)=f(x^{-1}tx)\De(x)^{1/2}$.  Of course, $\pi'(G)$ is not usually contained in $M$.  In case (c), $\al_{x}(\phi)=\pi_{2}(x)\phi\pi_{2}(x^{-1})$ where $\pi$ is the left regular representation of $G$, and (b) is similar.

If the action is inner on $M$, then $M^{*}, M_{*}$ are $G$-modules; for example, for $a\in M_{*}$, $\left<xm,a\right> = \left<m,ax\right>$ and similarly for $xa$.  Note that the maps $x\to xa, x\to ax$ are norm continuous so that $M_{*}$ is a Banach $G$-module, and hence a Banach $L^{1}(G)$-module.   Dualizing gives that $M, M^{*}$ are Banach $L^{1}(G)$-modules.  In the inner case, we have to distinguish carefully between $f.a$ and $fa$, and $m.f$ and $mf$ ($f\in L^{1}(G)$).   

Now suppose that $G$ has an action $\al$ (not necessarily inner) on $M$.  
\begin{definition}  \label{def:Gamenable}
The von Neumann algebra $M$ is called {\em $G$-amenable} (for $\al$) if there exists $p\in S(M)$ such that $(\al_{x^{-1}})^{*}(p)=x.p=p$ for all $x\in G$. Such a state will be called a  {\em $G$-invariant state}.  
\end{definition}

{\em $G$-amenability} in our terminology corresponds to {\em $G$ acting amenably on $M_{*}$} in Stokke's terminology (\cite{Stokke}).  
For the examples (a), (b) and (c) above where $M=L^{\infty}(G)$, in (a), a $G$-invariant state is respectively an inner invariant mean, a right invariant mean and a left invariant mean.  So $G$-amenability is the same as {\em inner amenability} (\cite{LauP}) for (a), and to amenability for (b) and (c).  In the case where the action of $G$ on $M=B(\mfH)$ is determined by a representation $\pi$ of $G$, $G$-amenability is the same as the amenability of the representation (\cite{Bekka}).  (These examples are also noted by Stokke within his framework.) 

$G$-amenability can be characterized in terms of a Dixmier type property.  (For the classical Dixmier property, see \cite[Proposition 2.1]{Paterson}.)  The corresponding property for $G$-amenability is easy to prove.  This says that {\em there exists a $G$-invariant state on $M$ if and only if the linear span $Y$ of operators of the form $(\al_{x}(m)-m)$ ($x\in G, m\in M$) is not norm dense in $M$.}  The proof is effectively the same as that in the classical case.  One implication is trivial.  For the converse, suppose that $Y$ is not norm dense in $M$ and let $p\in M^{*}$ be non-zero and such that $p(Y)=0$.  We can suppose that $p$ is self-adjoint and then, using the uniqueness of the Jordan decomposition (\cite[p.45]{Pedersen}, \cite[pp.120,140]{Takv1}), that $p$ is positive.  The result now follows.  

The following proposition follows from \cite[Corollary 1.12]{Stokke} - for the sake of completeness, we sketch the proof.

\begin{proposition}      \label{prop:amen} 
The following are equivalent:
\bi
\item[(i)] $M$ is $G$-amenable.
\item[(ii)] There exists a state $q$ on $M$ such that $q(m.\mu)=q(m)$ for all $m\in M, \mu\in P(G)$.  (So $q$ is {\em topological invariant}.)
\item[(iii)] Given $C\in \mathcal{C}(G)$ and $\eps>0$, there exists $a\in S_{*}(M)$ such that for all $x\in C$, we have 
\begin{equation}    \label{eq:reiter1}
\norm{x.a-a}<\eps.
\end{equation}
\ei
\end{proposition}    
\begin{proof}
(iii) $\Rightarrow$ (i) Write the $a$ of (\ref{eq:reiter1}) as $a_{C,\eps}$.  Then in the natural way, $\{a_{C,\eps}\}$ is a net, and any weak$^{*}$-cluster point of this net in $S(M)$ is a $G$-invariant state.\\
(i) $\Rightarrow$ (ii) Let $p$ be an invariant state.  As in the proof of \cite[Proposition 1.7]{Paterson}, for $m\geq 0$ in $M$, the map $\nu\to p(m.\nu)$ is constant for all $\nu\in P(G)$, and for fixed $\nu_{0}\in P(G)$, the state $q$, where $q(n)=p(n.\nu_{0})$, is {\em topologically} invariant, i.e. for all $\mu\in P(G)$, $\mu .q=q$. \\ 
(ii) $\Rightarrow$ (iii) There exists a net $\{t_{\bt}\}$ in $S_{*}(M)$ such that $\widehat{t_{\bt}}\to q$ weak$^{*}$ in $M^{*}$, and the argument of Namioka (\cite[p.8]{Paterson}) then applies to give a net $\{a_{\de}\}$ in $S_{*}(M)$ such that $\norm{\mu .a_{\de}-a_{\de}}\to 0$ for all $\mu\in P(G)$.  Now let $C\in \mathcal{C}(G)$.  With $x_{i}$ $(1\leq i\leq n)$, $U, g, \eta>0$ as in the proof of \cite[Theorem 4.4]{Paterson}, choose $b\in S_{*}(M)$ so that 
$\norm{(x_{i}*g).b - b}<\eta$ for $1\leq i\leq n$.  Then as in the proof of that theorem, $a=g.b$ satisfies (\ref{eq:reiter1}). 
\end{proof}

%Let $UC(M)$ be the set of elements $m.f$ where $m\in M, f\in L^{1}(G)$.  Then by Cohen's factorization theorem (cf. %\cite[p.190]{Bekka}) $UC(M)$ is a unital $C^{*}$-subalgebra of $M$.  Further, it is an $L^{1}(G)$-submodule of $M$ and is %exactly the space of elements $m\in M$ for which the map $x\to \al_{x}(m)$ is norm continuous.  We now state a result of %Stokke (\cite[Proposition 1.10]{Stokke}): it is proved as for the corresponding result in the case $M=B(\mfH)$ by Bekka %(\cite[Theorem 3.4]{Bekka}).  A state $q$ on $M$ is called {\em topological invariant} if $q(m.f)=q(m)$ for all 
%$m\in M, f\in P(G)$.

%\begin{theorem}    \label{th:uchthing}
%$M$ is $G$-amenable if and only if there exists a topological invariant state on $UC(M)$.
%\end{theorem}

We now discuss another kind of state on $M$ suggested by the work of Connes  (\cite[p.109]{Connes}).  We do not strictly need these states for the purposes of this paper but feel that they are of intrinsic interest. Suppose that the action 
$x\to \al_{x}$ of $G$ on $M$ is inner.  From \cite[Theorem 32.22]{HR2}, the sets
$L^{1}(G)M, ML^{1}(G)$ are normed closed subspaces of $M$.  Let $UC(G,M)$ be the norm closure of $L^{1}(G)M + ML^{1}(G)$ in $M$.  Obviously, $UC(G,M)$ is an $L^{1}(G)$-submodule of $M$ and also a $G$-submodule of $M$.  A state $p$ on $M$ is called a {\em $G$-hypertrace} (or just {\em hypertrace} when $G$ is understood) if $p(mf-fm)=0$ for all $m\in M, f\in L^{1}(G)$.  

\begin{proposition}    \label{prop:hyperamen}
Let the action $x\to \al_{x}$ of $G$ on $M$ be inner.   If there exists a $G$-hypertrace on $M$ then there exists a state on $M$ which is $G$-invariant on $UC(G,M)$.  If $M$ is $G$-amenable then there exists a $G$-hypertrace on $M$.  
\end{proposition}
\begin{proof} Suppose first that there is a hypertrace on $M$.  Let $x_{i}$ ($1\leq i\leq n$) belong to $G$, $e$ one of the $x_{i}$'s, and let $f_{j}$ ($1\leq j\leq m$) belong to $C_{c}(G)$.  Let $\eps>0$.  By Namioka's argument, there exists 
$a\in S_{*}(M)$ such that 
\[  \norm{ax_{i}f_{j} - x_{i}f_{j}a}<\eps, \hspace{.2in} \norm{f_{j}x_{i}a - af_{j}x_{i}}<\eps         \]
for $1\leq i\leq n, 1\leq j\leq m$.  Then for any $i,j$,
\[  \norm{(ax_{i} - x_{i}a)f_{j}}=\norm{[a(x_{i}f_{j}) - (x_{i}f_{j})a] - x_{i}[af_{j} - f_{j}a]}
\leq 2\eps.\]
A similar argument applies to $\norm{f_{j}(ax_{i} - x_{i}a)}$.  It follows that there is a net $\{a_{\de}\}$ in 
$S_{*}(M)$ such that $\widehat{a_{\de}}\to p$ weak$^{*}$ for some $p\in M^{*}$ and such that 
$\norm{(a_{\de}x - xa_{\de})f}\to 0$, $\norm{f(a_{\de}x - xa_{\de})}\to 0$ for all $x\in G, f\in C_{c}(G)$.  Then for 
$m\in M$, $\left|p((mf)x - x(mf))\right|= \lim_{\de}\left|\widehat{a_{\de}}((mf)x - x(mf))\right|
=\lim_{\de}\left|(fxa_{\de} - fa_{\de}x)(m)\right|\leq \norm{fxa_{\de} - fa_{\de}x}\norm{m}\to 0$.  Similarly, \
$p((fm)x - x(fm))=0$.  So for $n\in UC(G,M)$, we have $p(nx-xn)=0$.  So there exists a $G$-invariant state on $UC(G,M)$.

Suppose now that $M$ is $G$-amenable.  Using Reiter's condition (\ref{eq:reiter1}), there exists a net $\{a_{\de}\}$ in 
$S_{*}(M)$ such that $\widehat{a_{\de}}\to q$ for some $G$-invariant state $q$ on $M$ and such that 
$\norm{xa_{\de} - a_{\de}x}\to 0$ uniformly on compact subsets of $G$.  Let $f\in C_{c}(G)$ with compact support $C$.  Let $\eps>0$.  Then there exists $\de_{0}$ such that for all $\de\geq \de_{0}$ and all $x\in C$, we have
$\norm{xa_{\de} - a_{\de}x}<\eps$.  Then for all $\de\geq \de_{0}$,
\beqns
\norm{fa_{\de} - a_{\de}f} & = & \norm{\int_{C} f(x)[xa_{\de} - a_{\de}x]\,d\la(x)}   \\
&\leq & \int_{C}\left|f(x)\right|\norm{xa_{\de} - a_{\de}x}\,d\la(x)   \\
&\leq & \eps\norm{f}_{1}.
\eeqns
So for all $f\in C_{c}(G)$, we have $\norm{fa_{\de} - a_{\de}f}\to 0$.  So for $m\in M$,  we have
$\left|q(fm - mf)\right|=\left|\lim(fm - mf)(a_{\de})\right| = 
\lim\left|m(a_{\de}f - fa_{\de})\right|=0$. So $q$ is a hypertrace on $M$.
\end{proof} 

\noindent
{\bf Note} For a discrete group $G$, it is trivially true that $UC(G,M)=M$, and so the existence of a hypertrace on $M$ is equivalent to $G$-amenability for $M$.  In fact, in that case, a state on $M$ is a $G$-hypertrace if and only if it is $G$-invariant.  But in general, the existence of a hypertrace on $M$ is strictly weaker than $G$-amenability.  For example, consider the case where $M=L^{\infty}(G)$ under the conjugation action.  Then every $(mf-fm)$ is bounded and continuous on $G$.  Let $\de_{e}$ be the point mass at $e$.  Then 
$\de_{e}(mf-fm)=m(f)-m(f)=0$, and so any extension of the point mass $\de_{e}$ on $C(G)$ to a state on $L^{\infty}(G)$ is a hypertrace on $M$.  But $M$ need not be $G$-amenable.  (For example, by \cite{Losrind}, any connected, non-amenable $G$ is not inner amenable.)

%*****************
%*****************
%*****************

\section{Fixed-point characterizations for $G$-amenability}%4

There seems in general to be no hope of a fixed point theorem of the classical kind for $G$-amenability.  Indeed, consider the special case where $M=L^{\infty}(G)$ and the action of $G$ is given by conjugation (\S 3).  Then $G$-amenability for $M$ is just inner amenability for $G$.  A classical fixed point theorem in that case would presumably be of the form: {\em if $G$ is inner amenable and X is a certain kind of right Banach $G$-module, then every left invariant, weak$^{*}$ compact convex subset $K$ of $X^{*}$ has a $G$-fixed point}.  The case $X=L^{\infty}(G)$ ($G$ inner amenable), with conjugation as the action, would be a natural choice for such a right Banach $G$-module since there is an inner invariant mean in $X^{*}$.  However, in that situation, there can exist $K$'s in $X^{*}$ that don't contain a $G$-fixed point.  For example, suppose that G is discrete and that the only inner invariant mean on
$\ell^{\infty}(G)$ is the trivial one $\delta_{e}$ (e.g. $G$ could be $F_{2}$ (\cite{Effros})).  Then we could take K to be the set of means $m\in X^{*}$ for which $m(\chi_{e})=0$.  So we have to look in a different direction for the appropriate version of a fixed-point theorem for $G$-amenability.  We adapt the approach of \cite{LauP} for inner amenability.

Now let $M$ be a von Neumann algebra with an action $x\to \al_{x}$ of $G$ (as in \S 3.)  Let $A=M_{*}$ and 
$\bt_{x}(a)=(\al_{x^{-1}})_{*}$.  We now describe the class of $M$'s for which we are able to give a fixed-point characterization of $G$-amenability.  We will require the following properties (a), (b) for $A$:\\
\bi
\item[(a)] $A$ is a Banach algebra;
\item[(b)] $\bt_{x}$ is an isomorphism of $A$ for each $x\in G$.
\ei
So given these properties, the map $x\to \bt_{x}$ is an action of $G$ on $A$.  We note that if $\bt$ is inner on $A$, then the action $\al$ of $G$ on $M$ is also inner.

Now assume that $A$ satisfies (a) and (b), and let $X$ be a left Banach $G-A$-module (\S 2).  Then in the usual way, $X^{**}$ is a left $G$-module: for $\xi\in X, f\in X^{*}, n\in X^{**}$, we have $xn(f)=n(fx)$ where $fx(\xi)=f(x\xi)$.  So for each $x\in G$, there is a bounded operator $T_{x}\in B(X^{**})$ given by: $T_{x}m=xm$.  Also, since $X$ is a left Banach $A$-module, there is for each $a\in A$, an operator $T_{a}\in B(X^{**})$ such that $T_{a}n=an=\hat{a}n$.  (As we saw in \S 2, $T_{a}$ makes sense in the natural way for $a\in A^{**}$.)  The next result follows by dualizing the equality 
$xax^{-1}\xi=\bt_{x}(a)\xi$.  

\begin{proposition}        \label{prop:alxta}
For all $x\in G, a\in A$,  
\begin{equation}          \label{eq:alxta}
T_{x}T_{a}T_{x^{-1}}=T_{\bt_{x}(a)}.
\end{equation}
\end{proposition}

Now canonically, $B(X^{**})$ is the dual Banach space $(X^{**}\widehat{\otimes}X^{*})^{*}$: $T\in B(X^{**})$ is identified with the functional on $X^{**}\widehat{\otimes}X^{*}$ determined by: $m\otimes w\to \left<Tm,w\right>$.  Let $P_{X}$ be the weak$^{*}$ closure in $B(X^{**})$ of $\{T_{a}: a\in S_{*}(M)\}$.  

\begin{definition}            \label{def:fpp}
The Banach algebra $A$ has {\em the $G$-fixed point property} if for any left Banach $G-A$-module $X$, there exists $T_{0}\in P_{X}$ such that for all $x\in G$,
\begin{equation}        \label{eq:tg0}
T_{x}T_{0}T_{x^{-1}}=T_{0}.    
\end{equation}
A left Banach $G-A$-module $X_{0}$ is said to have a {\em $G-A$-fixed point} if there exists $T_{0}\in P_{X_{0}}$ such that (\ref{eq:tg0}) holds.
\end{definition}

What will correspond to the fixed-point property for $G$-amenability for $M$ will be the $G$-fixed point property for $A$.
In connection with this, it is helpful to consider the classical amenability fixed-point theorems, e.g. \cite[Chapter 2]{Paterson}.  As an illustration, $G$ is amenable $\iff$ every affine, jointly continuous compact convex $G$-set has a $G$-fixed point $\iff$ there is a $G$-invariant mean on $L^{\infty}(G)$ $\iff$ there is a $G$-invariant mean on $LUC(G)$. 
There is a similar pattern for the $G$-fixed point property.  Amenability corresponds to $G$-amenability and the jointly continuous result to the $G$-fixed point property for $A$.   We now look for special $X$'s in the $G$-amenable case corresponding to $(L^{\infty}(G))_{*}=L^{1}(G)$ and to a substitute for (the non-existent) ``$(LUC(G))_{*}$''.  For the first, we take the ``largest'' natural Banach algebra that is a left Banach $G-A$-module; this is $L^{1}(G,A^{1})$.  For the second, we take the ``smallest'' natural Banach algebra that is a left Banach $G-A$-module.  This is $A^{1}$.  In both cases, we require the existence of a $G-A$-fixed point.  One might ask: why not use $A$ in place of $A^{1}$?  The reason why we require the identity to be adjoined to $A$ is that the existence of a $G-A$-fixed point for $A$ does not in general imply $G$-amenability for $M$.  For example, consider the case where $M$ is general and we take $A^{2}=\{0\}$: then $A$ is a left Banach $G-A$-module, and (a) and (b) are satisfied.   Trivially,  (\ref{eq:tg0}) holds for $X=A$ with $T_{0}=0$, so that $A$ has a $G-A$-fixed point.  But $M$ need not be $G$-amenable.  

We note that $A^{1}$ is a left Banach $G-A$-module. Here $A^{1}$ is the Banach algebra obtained by adjoining an identity $1$ to $A$: so $\norm{a+c1}=\norm{a}+\left|c\right|$, and canonically, $A^{1}=A\oplus \C 1$.  The action $x\to \bt_{x}$ of $G$ on $A$ extends to an action, also denoted $x\to \bt_{x}$, on $A^{1}$ by taking $\bt_{x}(1)=1$ for all $x\in G$.   The left Banach $G$-module structure on $A^{1}$ is given by: $xb=\bt_{x}(b)$ (cf. \cite[Example 1.2]{Stokke}).  Then $A^{1}$ is a left Banach $G-A$-module since
\[       xax^{-1}b=x(a\bt_{x^{-1}}(b))=\bt_{x}[a\bt_{x^{-1}}(b)]=\bt_{x}(a)b.       \]
From \S 2, $L^{1}(G,A^{1})$ is also a left Banach $G-A^{1}$-module, and hence {\em a fortiori} a left Banach 
$G-A$-module.  We now give our fixed-point theorem.

\begin{theorem}   \label{th:fpt}
Let $A$ satisfy (a) and (b).  Then the following statements are equivalent:
\bi
\item[(i)] $M$ is $G$-amenable; 
\item[(ii)] $A$ has the $G$-fixed point property;
\item[(iii)] $A^{1}$ has a $G-A$-fixed point;
\item[(iv)] $L^{1}(G,A^{1})$ has a $G-A$-fixed point.
\ei  
\end{theorem}
\begin{proof}
We show: (i)$\Rightarrow$ (ii),  (ii)$\Rightarrow$ (iii), (iii)$\Rightarrow$ (i), (ii)$\Rightarrow$ (iv) and (iv)$\Rightarrow$ (i).

Suppose then that $M$ is $G$-amenable and let $X$ be a left Banach $G-A$-module.  By Proposition~\ref{prop:amen}, there exists a net $\{a_{\de}\}$ in $S_{*}(M)$ such that $\norm{\bt_{x}(a_{\de})-a_{\de}}\to 0$ for all $x\in G$.  We can assume that $T_{a_{\de}}\to T_{0}$ weak$^{*}$ for some $T_{0}\in P_{X}$.  Then as in a proof of Xu's (\cite[Theorem 11]{Xu}) for $m\in X^{**}$, and using Proposition~\ref{prop:alxta}, 
\[ \norm{T_{x}T_{a_{\de}}T_{x^{-1}}m - T_{a_{\de}}m}=\norm{\bt_{x}(a_{\de})m-a_{\de}m}\leq 
\norm{\bt_{x}(a_{\de})-a_{\de}}\norm{m}\to 0 \]
and (\ref{eq:tg0}) follows.  So $A$ has the $G$-fixed point property and (i)$\Rightarrow$ (ii).  The implication 
(ii)$\Rightarrow$ (iii) is trivial.

Now suppose that the left Banach $G-A$-module $Y=A^{1}$ has a $G-A$-fixed point.  Then there exists $T_{0}\in P_{Y}$ such that
$T_{x}T_{0}T_{x^{-1}}=T_{0}$ for all $x\in G$.  By assumption, there exists a net $\{T_{a_{\si}}\}$, where the $a_{\si}\in S_{*}(M)$, such that $T_{a_{\si}}\to T_{0}$ weak$^{*}$ in $B(Y^{**})$.  We can suppose that $\widehat{a_{\si}}\to n_{0}$ weak$^{*}$ for some $n_{0}\in S(M)$.  We now show that $T_{0}\hat{1}=n_{0}$.  To this end, we have $Y=A\oplus \C 1$, and so can identify $Y^{*}$ as the direct sum $A^{*}\oplus \C p$ where for $f\in A^{*}$, $b\in A$ and $\mu, \nu\in \C$, $f+\mu p\in (A^{1})^{*}$ is given by: $(f+\mu p)(b+\nu 1)=f(b) + \mu\nu 1$.  The norm on $Y^{*}$ is given by:
$\norm{f+\mu p}=\max\{\norm{f},\left|\mu\right|\}$.  Then $Y^{**}$ is canonically identified with the Banach space direct sum $A^{**}\oplus \C\hat{1}$.  We will regard $A^{*}, A^{**}$ as subspaces of $Y^{*}, Y^{**}$ respectively.  Now  
$\left<T_{0}\hat{1},f+\mu p\right>=\lim \left<\widehat{a_{\si}}\hat{1}, f+\mu p\right>
=\lim\left<\widehat{a_{\si}},\hat{1}(f+\mu p)\right>=n_{0}(\hat{1}(f+\mu p))$.  The equality $T_{0}\hat{1}=n_{0}$ will follow once we have shown that $\hat{1}(f+\mu p)=f\in A^{*}$.  For this, for any $a\in A$, 
$\hat{1}(f+\mu p)(a)=\hat{1}((f+\mu p)a)$ and $((f+\mu p)a)(b+\nu 1)=(f+\mu p)(ab+\nu a)=f(ab)+\nu f(a)$, and it follows that 
$(f+\mu p)a=fa + f(a)p$ (where $fa\in A^{*}$).  So $\hat{1}(f+\mu p)(a)=\widehat{1}(fa+f(a)p)=f(a)$ giving
$\hat{1}(f+\mu p)=f$ as required.  Last, 
\[  n_{0}=T_{0}\hat{1}=T_{x}T_{0}T_{x^{-1}}\hat{1}=T_{x}T_{0}\widehat{\bt_{x^{-1}}(1)}=T_{x}T_{0}\hat{1}=\bt_{x}^{**}(n_{0}) \]
and so $M$ is $G$-amenable.   So (iii)$\Rightarrow$ (i).  (ii)$\Rightarrow$ (iv) is trivial.

It remains to show that (iv) $\Rightarrow$ (i).   So assume that $X=L^{1}(G,A^{1})$ has a $G-A$-fixed point.  We recall 
(\S 2) that $X=L^{1}(G,A^{1})$ is a left Banach $G-A$-module under the following operations, where for $x,t\in G$, $f\in X$ and $a\in A$, we have $xf(t)=\bt_{x}(f(x^{-1}t))$ and $(af)(t)=a(f(t))$.  Since $A$ has the $G$-fixed point property, there exists $T_{0}\in P_{X}$ such that (\ref{eq:tg0}) holds for all $x$.  By assumption, there exists a net $\{T_{a_{\si}}\}$, where the $a_{\si}\in S_{*}(M)$, such that 
$T_{a_{\si}}\to T_{0}$ weak$^{*}$ in $B(X^{**})$.  Then for each $m\in X^{**}$ and $x\in G$, we have
\begin{equation}          \label{eq:asigma}
xa_{\si}x^{-1}m - a_{\si}m\to 0
\end{equation}
weak$^{*}$ in $X^{**}$.

Now define a map $Q:C_{c}(G,A^{1})\to A^{1}$ by setting: $Q(f)=\int f(t)\,d\la(t)$.  Then $Q$ is linear and continuous with norm $\leq 1$ and so extends to a linear map of norm $1$, also denoted by $Q:X\to A^{1}$.  Note that for $a\in A$, $Q(af)=aQ(f)$.  Identify $f\in X$ with its image $\hat{f}\in X^{**}$.  From (\ref{eq:asigma}) and (\ref{eq:btxax}), for each 
$x\in G$ and $f\in X$,
\[   \bt_{x}(a_{\si})f - a_{\si}f=xa_{\si}x^{-1}f - a_{\si}f\to 0        \]
weakly in $X$.  Using the Namioka argument, we can assume that $\norm{\bt_{x}(a_{\si})f - a_{\si}f}\to 0$ in $X$ for all $x\in G, f\in X$.  Choose $f=g\otimes 1$ where $g\in P(G)$.  Then $Q(f)=1$, and by the continuity of $Q$,
\[  (\bt_{x}(a_{\si}) - a_{\si})1 = Q(\bt_{x}(a_{\si})f - a_{\si}f)\to 0               \]
in norm in $A^{1}$.  So $(\bt_{x}(a_{\si}) - a_{\si})\to 0$ in norm in $A$.  Recalling that 
$\bt_{x}=(\al_{x^{-1}})_{*}$, we have $x.a(a)=\bt_{x}(a)$ for all $a\in A$, and applying Proposition~\ref{prop:amen} (with $G$ as discrete) we have that 
$M$ is $G$-amenable.
\end{proof}
\begin{corollary}          \label{cor:fpt}
$A$ has the $G$-fixed point property if and only if it has that property when $G$ is given the discrete topology.
\end{corollary}
\begin{proof}
$G$-amenability for $M$ is independent of the topology on $G$.
\end{proof}  

It is usually desirable to assume a bounded approximate identity in a Banach algebra $A$.  But, as we will see below, in important examples, the Banach algebra $A=M_{*}$ does not have a bounded approximate identity, and this is why, in Theorem~\ref{th:fpt}, a bounded approximate identity in $A$ is not assumed.   We also do not assume in that theorem that the action $\bt$ of $G$ on $A$ is inner.  Indeed, the inner condition gives that $A$ is a Banach $G$-module, but the $G$-module action does not extend in any natural way to its unitization $A^{1}$, which is used in the theorem.  On the other hand, the action $x\to \al_{x}$ {\em does} always extend trivially to $A^{1}$, and this is why we need only assume in the theorem the presence of a $G$-action on $A$. 

However, for some applications, we are in the situation where $A$ has a bounded approximate identity $\{e_{\de}\}$ and the action of $G$ on $A$ is inner.  We now show that in that situation, we do not have to unitize $A$, and the $G$-fixed point property is required only for {\em essential} left Banach $A$-modules.   (Recall (\S 2) that if $X$ is an essential left Banach $A$-module, then it is automatically a left Banach $G-A$-module.)  

\begin{theorem}     \label{th:inner}
Let $A$ satisfy (a) and (b), have a bounded approximate identity and be such that the action $\bt$ of $G$ on $A$ is inner.  Then the following statements are equivalent:
\bi
\item[(i)] $M$ is $G$-amenable;
\item[(ii)] $A$ has the $G$-fixed point property for essential left Banach $A$-modules;
\item[(iii)] $A$ has a $G-A$-fixed point;
\item[(iv)] $L^{1}(G,A)$ has a $G-A$-fixed point.
\ei
\end{theorem}
\begin{proof}
That (i) implies (ii) follows from Theorem~\ref{th:fpt}.  The implications (ii) implies (iii) and (ii) implies (iv) follow trivially, since both $A, L^{1}(G,A)$ are essential left Banach $A$-modules.  Next, we prove that (iii) implies (i).  By Proposition~\ref{prop:bna}, there exists $e\in A^{**}$ such that $me=m$ for all $m\in A^{**}$.  We can suppose that $\hat{a_{\si}}\to n_{0}$ weak$^{*}$ in $A^{**}$.  Then for $x\in G$, 
\[ x\hat{a_{\si}}x^{-1} - \hat{a_{\si}} = xa_{\si}x^{-1}e-a_{\si}e=(xT_{a_{\si}}x^{-1} - T_{a_{\si}})e\to 
xT_{0}x^{-1}e - T_{0}e=0  \]
weak$^{*}$ in $A^{**}$.  But  $(x\hat{a_{\si}}x^{-1} - \hat{a_{\si}}) \to (xn_{0}x^{-1} - n_{0})$ weak$^{*}$ in $A^{**}$.
So $n_{0}$ is a $G$-invariant state on $M$, and so $M$ is $G$-amenable.

We now prove that (iv) implies (i).  So assume that $X=L^{1}(G,A)$ has a $G-A$-fixed point.  Then there exists $T_{0}\in P_{X}$ such that (\ref{eq:tg0}) holds for all $x$.  As in the preceding paragraph, there exists a sequence $\{a_{\si}\}$ in $S_{*}(M)$ such that for all $x\in G, p\in X^{**}$, 
\begin{equation}  \label{eq:asip}
(xa_{\si}x^{-1} - a_{\si})p\to 0
\end{equation}
weak$^{*}$ in $X^{**}$.  Let $g\in P(G)$.  We can suppose that the bounded approximate identity $\{\widehat{e_{\de}}\}$ converges to $e$ weak$^{*}$ in $X^{**}$.  We can also suppose that 
$\widehat{g\otimes e_{\de}}\to e'$ weak$^{*}$ in $X^{**}$.  Then for each $\de$, $Q(g\otimes e_{\de})=e_{\de}$, and by the weak$^{*}$ continuity of $Q^{**}$, we get $Q^{**}(e')=e$.  Applying $Q^{**}$ to (\ref{eq:asip}) with $e'$ in place of $e$, and noting that $Q^{**}$ is a left $A$-module map, it follows that $xa_{\si}x^{-1}e-a_{\si}e \to 0$ weak$^{*}$ in $A^{**}$.  
(i) now follows as in the preceding paragraph.
\end{proof}

In a number of cases, the Banach algebra $A$ of Theorem~\ref{th:fpt} does not have a bounded approximate identity.  We saw a very simple example of this earlier for the case where $A^{2}=\{0\}$.  A more interesting example is that of the commutative Banach algebra $\ell^{1}(G)$ under pointwise product and with conjugation as the $G$-action.  This will be discussed in Example 1 of \S 5.   An example of special importance arises in Bekka's discussion of amenable representations (\cite{Bekka}).  Let $\pi$ be a unitary representation of $G$ on a Hilbert space $\mfH_{\pi}=\mfH$ and 
$M=B(\mfH)$.  We take $A$ to be the algebra $TC(\mfH)$ of trace class operators on $\mfH$ with the natural conjugation action: $\bt_{x}(T)=\pi(x)T\pi(x^{-1})$.   Then $A$ satisfies (a) and (b), and the action is obviously inner.  However, the ``inner action'' result Theorem~\ref{th:inner} does not apply in this case, though of course the more general result Theorem~\ref{th:fpt} does apply.  (See Example 4 of \S 5.)  The reason for this is that, for $\mfH$ an infinite dimensional Hilbert space, the Banach algebra $A=TC(\mfH)$ does not have a bounded approximate identity.  We have been unable to find this result in the literature, and now give two short proofs of it.  

The first proof, in the ``Arens product'' spirit above, goes as follows.  Suppose that $A$ does have a bounded approximate identity.  Since (\cite[p.52]{Pedersen})
$K(\mfH)^{*}=A, A^{*}=B(\mfH)$, we can write $A^{**}=A+Q$ where $Q=K(\mfH)^{\perp}$.  Then for 
$T\in A^{*}, a\in A$, we have $Ta\in K(\mfH)$, and it follows that $nT=0$ for all $n\in Q$.  Hence $pn=0$ for all $p\in A^{**}, n\in Q$.  Let $e$ be a right identity of $A^{**}$ (as in the proof of Theorem~\ref{th:inner}).  Writing $e=a_{1} + n_{1}$, where $a_{1}\in A, n_{1}\in Q$, we see that $a_{1}$ is a right identity for $A$, which entails $a_{1}=I$ and a contradiction.  

For the second proof, if $A$ has a bounded approximate identity then by Cohen's factorization theorem, $A^{2}=A$.  Now $A^{2}$ is contained in the Schatten-von Neumann ideal $C_{1/2}(\mfH)$ 
(since (\cite[XI.9.9]{DS2}) $C_{p_{1}}C_{p_{2}}\subset C_{p}$ whenever 
$1/p=1/p_{1}+1/p_{2}$, $p_{1}, p_{2}>0$).  This is a contradiction since $A=C_{1}(\mfH)$ is not a subset of $C_{1/2}(\mfH)$.  (Take, for example, a positive operator $T\in K(\mfH)$ with eigenvalues $\mu_{n}$ for which $\{\mu_{n}\}\in \ell^{1}\setminus \ell^{1/2}$, e.g. $\mu_{n}=\frac{1}{n^{2}}$.)  
%%%%%%%%%%%%%
%%%%%%%%%%%%%
%%%%%%%%%%%%%

\vspace*{-.2in}
\section{Examples}%5
We now illustrate the theory of the preceding section with some examples and comments.  As discussed in the preceding section, a technical difficulty arises if $A$ does not have a bounded approximate identity.  This was resolved by using $A^{1}$ instead of $A$. If we stay with $X=A$ and neither require the action $\bt$ of $G$ on $A$ to be inner nor that there exist a bounded approximate identity in $A$, then we no longer have available the element $e$ of the proof of Theorem~\ref{th:inner}.  So the argument for (iii) implies (i) in that theorem fails.  This was illustrated by the easy example in which $A^{2}=\{0\}$ and $M$ is not $G$-amenable.  We now look at a more interesting example of this phenomenon; in this case, $M=\ell^{\infty}(G)$ is $G$-amenable.

\vspace*{.2in}
\noindent
{\bf Example 1}\\
Let $G$ be an infinite discrete group, $M=\ell^{\infty}(G)$, and $A=\ell^{1}(G)$ with {\em pointwise} multiplication.  If $G$ is not finite, then $A$ does not have a bounded approximate identity.  (If it did, then because $A^{2}=\ell^{1/2}(G)\neq A$, we would contradict Cohen's theorem.)  The action $x\to \al_{x}$ of $G$ on $M$ is given by conjugation, 
$\al_{x}\phi=x\phi x^{-1}$, and for $f\in \ell^{1}(G)=M_{*}$, 
$\bt_{x}f(t)=f(x^{-1}tx)$.  Next (a) and (b) of \S 4 are satisfied, and $A$ is a left Banach $G-A$-module, with left $G$-action given by: $xf=\bt_{x}(f)$.  
For $f\in A$, let $f_{b}\in M=\ell^{\infty}(G)$ be $f$ regarded as a bounded function on $G$.  The map $f\to f_{b}$ is norm continuous from $A$ into $A^{*}=M$.  Let 
$\ell^{1}(G)_{b}$ be the range of the map $f\to f_{b}$.

For $\phi\in \ell^{\infty}(G), f,g\in \ell^{1}(G)$, we have $\phi
f(g)=\sum_{x\in G}\phi(x)f(x)g(x)$, so that $\phi f\in \ell^{\infty}(G)$ is just
pointwise multiplication $\phi f_{b}$ of $\phi$ and $f_{b}$, and so 
equals $(\phi f)_{b}\in \ell^{1}(G)_{b}$.  For $x\in G$, define $\chi_{x}\in \ell^{1}(G)_{b}$ by: 
$\chi_{x}(y)=1$ if $y=x$ and is $0$ otherwise.  Then $\phi f=\sum \phi(x)f(x)\chi_{x}\in \ell^{\infty}(G)$.
If $m\in \ell^{1}(G)^{**}$, let
$\tilde{m}\in \ell^{\infty}(G)$ be given by: $\tilde{m}(x)=m(\chi_{x})$.  Note that given distinct elements $x_{1}, .., x_{n}$ in $G$, we can choose modulus 1
scalars $z_{i}$ such that
$\norm{m}\geq \mid m(\sum z_{i}\chi_{x_{i}})\mid=\sum \mid
\tilde{m}(x_{i})\mid$ so that $\tilde{m}\in \ell^{1}(G)_{b}$. 
Then $m\phi(f)=m(\phi f_{b})=\sum \phi(x)f(x)m(\chi_{x})$ giving
$m\phi=\tilde{m}\phi$ (pointwise product) which is in
$\ell^{1}(G)_{b}$.

Assume now in addition that the only inner invariant mean on $\ell^{\infty}(G)$ is the trivial one $\de_{e}$.  Let $n\in S(M)\subset A^{**}$ where
$n(\ell^{1}(G)_{b})=\{0\}$.  Then $T_{n}\in P_{A}$, where $T_{n}(m)=nm$, and
$T_{n}m(\phi)=n(m^{~}\phi)=0$, so that $T_{n}=0$ and trivially
$T_{x}T_{n}T_{x^{-1}}=T_{n}$ for all $x\in G$.  But since $\de_{e}$ is the only inner invariant mean on $\ell^{\infty}(G)$ and $n(\chi_{e})=0$, we do not have
$\bt_{x}^{**}n=n$ for all $x\in G$.  So the argument for (iii) $\Rightarrow$ (i) of Theorem~\ref{th:fpt}, with $A$ in place of $A^{1}$, fails.
However $M$ is $G$-amenable since $\de_{e}$ is a $G$-invariant state, and so $A$ has the $G$-fixed point property.

\vspace*{.2in}
\noindent
{\bf Example 2}\\
Suppose that $A$ satisfies (a) and (b) and is unital.  Then since $\bt_{x}(1)=1$ for all $x\in G$, we get that 
$\hat{1}$ is a non-zero invariant linear functional on $M$.  It follows from the comments after 
Definition~\ref{def:Gamenable} that $M$ is $G$-amenable.

\vspace*{.2in}
\noindent
{\bf Example 3}\\
Let $M=L^{\infty}(G)$ with $\al_{x}\phi(t)=x\phi x^{-1}$ and $A$ the Banach algebra $L^{1}(G)$ under convolution.  So, as earlier, $\bt_{x}f=x*f*x^{-1}$ and the action is inner on $A$.  Then  $A$ satisfies the hypotheses of 
Theorem~\ref{th:inner}.  A $G$-invariant state $p$ is just an inner invariant mean on $M$.  In the notation of \cite{LauP}, inner amenability of $G$ just means that there exists an inner invariant mean on $L^{\infty}(G)$.  
Further, $S_{*}(M)=P(G)$, and every left Banach $G$-module is, in a natural way, a left Banach $G-A$-module.  It follows using Theorem~\ref{th:fpt} and Theorem~\ref{th:inner} that {\em $G$ is inner amenable if and only if, whenever $X$ is a left Banach $G$-module, there exists $T_{0}\in P_{X}$ such that $T_{x}T_{0}T_{x^{-1}}=T_{0}$ for all $x\in G$.}   This is just \cite[Theorem 5.1]{LauP}.

\vspace*{.2in}
\noindent
{\bf Example 4}\\
Let $\pi$ be a unitary representation of $G$ on a Hilbert space $\mfH_{\pi}$, and let $M=B(\mfH_{\pi})$.  Define
$\al_{x}(T)=\pi(x)T\pi(x^{-1})$.  Then 
$M_{*}=A=TC(\mfH_{\pi})$, the duality between $M$ and $A$ being given by: $\left<T,a\right>=Tr(Ta)$.  Then 
$\bt_{x}(a)=\pi(x)a\pi(x)^{-1}$ and, of course, $A$ is a Banach algebra for which (a) and (b) are satisfied.  Bekka, in his theory of amenable representations, defines the representation $\pi$ to be {\em amenable} if (in our terminology) $B(\mfH_{\pi})$ is $G$-amenable.  He asked (\cite[p.400]{Bekka}) if there is an analogue of the fixed point property for amenable representations.  Using Theorem~\ref{th:fpt}, we can give an answer to this question: {\em $\pi$ is amenable if and only if $TC(\mfH_{\pi})$ has the $G$-fixed point property.} 

\vspace*{.2in}
\noindent
{\bf Example 5}\\
In this example we take $M=VN(G)\subset B(L^{2}(G))$ and $A=A(G)=M_{*}$, the Fourier algebra of $G$.  (This example will be generalized in the next section.)  Let $\al$ be the conjugation action of $G$ on $VN(G)$: so 
$\al_{x}(m)=\pi_{2}(x)m\pi_{2}(x^{-1})$ for all $m\in M$, where $\pi_{2}$ is the left regular representation of $G$ on $L^{2}(G)$.   It is easy to check that for $f\in A(G)\subset C(G)$, 
$\bt_{x}(f)(y)=f(x^{-1}yx)$, and that (a) and (b) are satisfied.   So $A(G)$ is $G$-amenable if and only if it has the $G$-fixed point property.   Since the action of $G$ on $M$ is inner, these properties are also equivalent 
(cf. the proof of Proposition~\ref{prop:hyperamen}) to the presence of a net $\{e_{\de}\}$ in    
$S_{*}(M)=P(A(G))=\{f\in A(G): f\mbox{ positive definite}, f(e)=1\}$ with the property that $\norm{xe_{\de}-e_{\de}x}\to 0$ for all $x\in G$, where, for $f\in A(G)$,  $xf(y)=f(yx), fx(y)=f(xy)$.  

It is of interest to determine for which groups $G$, $A(G)^{*}=VN(G)$ is $G$-amenable.  Here are some partial results in this direction.  First, if $G$ is amenable, then $VN(G)$ is $G$-amenable.  Indeed, in that case, there is an invariant state on $B(L^{2}(G))$ and hence by restriction on $VN(G)$.  However, there are many examples of non-amenable groups $G$ for which $VN(G)$ is $G$-amenable.  To illustrate this, if $G$ is an [IN]-group (for example, discrete) then $VN(G)$ has a natural tracial state and so is $G$-amenable.  However, $VN(G)$ is not always $G$-amenable.  For example, let $G$ be a non-amenable, separable, connected locally compact group.  Then (\cite[Corollary 6.9, (c)]{Connes}) $VN(G)$ is injective.  So there exists  a norm one projection $P:B(L^{2}(G))\to VN(G)$.  Note that (\cite[p.131]{Takv1}) $P$ is unital, positive 
and $P(vTw)=vP(T)w$ for all $T\in B(L^{2}(G))$ and all $v,w\in VN(G)$.  If $VN(G)$ is $G$-amenable, then there exists an invariant state $m$ on $VN(G)$, and it follows that $m\circ P$ is an invariant state on $B(L^{2}(G))$.  
So $G$ is amenable (\cite[Theorem 2.2]{Bekka}) - just
restrict the invariant state to the image of $L^{\infty}(G)$ in $B(L^{2}(G))$ - and we obtain a contradiction.  So $VN(G)$ is not $G$-amenable in that case.  

%*****************
%*****************
%*****************
\section{The fixed-point property for preduals of Hopf-von Neumann algebras}%6

In Example 5 of the preceding section, the von Neumann algebra $M$ was a {\em Kac algebra}.  This raises the natural question of the fixed point property for Kac algebras and indeed, more generally, for Hopf-von Neumann algebras.  We will investigate this question in the present section.  The group involved will be the {\em intrinsic group} of $M$, and $A=M_{*}$ has a natural Banach algebra structure satisfying (a) and (b) of \S 4, and, as we shall see, Theorem~\ref{th:fpt} applies.  In the Kac algebra case, this group is locally compact but this is not usually true in the Hopf-von Neumann case.  

Let $M$ be a Hopf-von Neumann algebra (\cite{Takesaki,Enock}) on a Hilbert space $\mfH$.  So (\cite[p.13]{Enock}) $M$ is a von Neumann algebra for which there is given an injective, normal morphism, called the {\em comultiplication}, 
$\Ga:M\to M\otimes M$ (the spatial tensor product of $M$ with itself) such that 
$(\Ga\otimes \imath)\circ \Ga=(\imath\otimes \Ga)\circ \Ga$ where $\imath:M\to M$ is the identity map.  Further, it is assumed that $\Ga(1)=1\otimes 1$.  Let $G$ be (\cite[p.14]{Enock}) the {\em intrinsic group} of $M$.  By definition, $G$ is the set of invertible elements $x$ of $M$ such that $\Ga(x)=x\otimes x$.  Then $G$ is a subgroup of the unitary group of $M$.  Further, the weak, ultraweak, strong and ultrastrong topologies coincide on $G$ and with this topology, $G$ is a topological group.  (If $M$ is a Kac algebra, then $G$ is (\cite[p.109]{Enock}) identified with the set of characters on $A$.)   We use the natural action of $G$ on $M$: $\al_{x}(m)=xmx^{-1}$.  Note that this action is inner on $M$.  Let $A=M_{*}$.

\begin{proposition}    \label{prop:aandb}
$A$ satisfies conditions (a) and (b) of \S 4.
\end{proposition}
\begin{proof}  
It is well known (\cite[p.13]{Enock}) that a Banach algebra multiplication on $A$ is determined by: $m(ab)=\Ga(m)(a\otimes b)$.  So (a) of \S 4 is satisfied.  Now let $m\in M, a,b\in A$ and $x\in G$.  Then $\left<m,x(ab)x^{-1}\right>=\left<x^{-1}mx,ab\right>=
\left<\Ga(x^{-1}mx),a\otimes b\right>=\left<\Ga(x)^{-1}\Ga(m)\Ga(x),a\otimes b\right>=\\
\left<\Ga(m),(x\otimes x)(a\otimes b)(x^{-1}\otimes x^{-1})\right>=\left<\Ga(m),(xax^{-1})\otimes (xbx^{-1})\right>\\
=\left<m,(xax^{-1})(xbx^{-1})\right>$, so that $\bt_{x}(ab)=\bt_{x}(a)\bt_{x}(b)$. (b) now follows.
\end{proof} 

Two canonical examples of Hopf-von Neumann algebras - indeed, they are even Kac algebras - are $L^{\infty}(G)$ and $VN(G)$.  In the case of $L^{\infty}(G)$, realized as multiplication operators on $L^{2}(G)$, the map $\Ga:L^{\infty}(G)\to L^{\infty}(G)\otimes L^{\infty}(G)=L^{\infty}(G\x G)$ is given by:
$\Ga(m)(s,t)=m(st)$ (\cite[p.10, p.55f.]{Enock}).  Of course, in that case, $A=L^{1}(G)$ under the convolution product.  The other example (\cite[p.95, p.114]{Enock}) is $VN(G)$ with $\Ga:VN(G)\to VN(G)\otimes VN(G)$ determined by: $\Ga(s)=s\otimes s$, where $s\in G$ (identified with $\pi_{2}(s)\in VN(G)$).

In the first case, $M=L^{\infty}(G)$, the intrinsic group of $M$, denoted here by $G_{i}$ (to avoid confusion), is the topological group of characters on $G$.  So $G_{i}=\widehat{G}$ and for $\ga\in \widehat{G}$, $\ga m\ga^{-1}=m$ for all $m\in L^{\infty}(G)$.   Any state on $L^{\infty}(G)$ is then $G_{i}$-invariant, so that $L^{\infty}(G)$ is trivially $\widehat{G}$-amenable.   In the second case, $M$ is the Kac algebra $VN(G)$ so that $A=A(G)$.  The intrinsic group of $VN(G)$ is just $G\cong \pi_{2}(G)$ (\cite[p.136]{Enock}), and so 
$\al_{x}(m)=xmx^{-1}$.  This case is discussed in Example 5 of the preceding section. 

Now let $M$ be again a general Hopf-von Neumann algebra.  We now discuss the $G$-fixed point property for $A$.  Since $G$ is not necessarily a locally compact group, we cannot use formulations of the fixed point property involving $L^{1}(G)$.  Instead, we use Theorem~\ref{th:fpt} in this case with $G$ taken to be {\em discrete}.  Note that the action $\bt$ of $G$ on $A$ is not usually inner so that Theorem~\ref{th:inner} does not apply.   Indeed, for $\bt$ to be inner, we would require in particular that for each $x\in G$, $a,b\in A$, the multiplier condition
$x(ab)=(xa)b$ holds.  However it is readily checked that instead, $x(ab)=(xa)(xb)$.

%*****************
%*****************
%*****************
\section{F\o lner conditions for $G$-amenable von Neumann algebras}%7

$G$-amenability for von Neumann algebras is interestingly connected with {\em traces}, as will be clear from the study of the F\o lner conditions in this section.  For the present, we illustrate the connection with an easy result (cf. \cite[Proposition 1]{Valette}).

\begin{proposition}           \label{prop:piamen}
Let $\pi$ be a unitary representation of a locally compact group $G$ on a Hilbert space $\mfH$ and suppose that $\pi(G)$ is contained in a von Neumann subalgebra $N$ of $B(\mfH)$ such that $N$ is amenable and admits a tracial state.  Then $\pi$ is amenable (i.e. $B(\mfH)$ is $G$-amenable).
\end{proposition}
\begin{proof}
As in Example 5, \S 5, there exists a $G$-invariant state on $B(\mfH)$.   So $\pi$ is amenable.
\end{proof}

The F\o lner condition for a locally compact group $G$ is as follows (\cite[Theorem 4.10]{Paterson}): {\em given $\eps>0$ and  $C\in \mathcal{C}(G)$, there exists a nonnull, compact subset $K$ of $G$ such that $\la(xK\bigtriangleup K)<\eps\la(K)$
for all $x\in C$.}  Connes (\cite[Theorem 5.1]{Connes}) and Bekka (\cite[Theorem 6.2]{Bekka}) established what can be regarded as operator generalizations of the F\o lner condition.  In their context, the characteristic function $\chi_{K}$ of the above F\o lner condition is replaced by a projection $P$, and the finiteness of $\la(K)$ is replaced by the finite rank property for $P$.  Bekka says that his result was largely inspired by that of Connes, and our results are effectively extensions of theirs and proved in a similar way.  We start by discussing the F\o lner conditions of Connes and Bekka.

Connes's F\o lner condition is used in the proof of his famous theorem on the uniqueness of the hyperfinite $II_{1}$ factor.
In that theorem, it was shown that seven properties for a $II_{1}$ factor are equivalent, the really hard part being to go from 7. to 1..  What we need is the equivalence of 7. and 6., and this can be conveniently formulated as follows.  We will refer to this theorem as Theorem A.

(Theorem A) Let $M$ be a $II_{1}$ factor acting on $\mfH=L^{2}(M,\tau)$, where $\tau$ is the unique normal tracial state on $M$. Let $U(M)$ be the unitary group of $M$ regarded as discrete, and $\norm{.}_{HS}, \left<.\right>_{HS}$ be respectively the norm and inner product associated with the space of Hilbert-Schmidt operators on $\mfH$.  Then 1) and 2) are equivalent:
\bi
\item[1)] $B(\mfH)$ is $U(M)$-amenable (or equivalently, there exists a $U(M)$-hyper\-trace on $B(\mfH)$).
\item[2)] Given $u_{1},\ldots, u_{n}\in U(M)$ and $\eps >0$, there exists a non-zero, finite rank projection $P\in B(\mfH)$ such that for all $j$, $1\leq j\leq n$, 
\bi
\item[(a)] $\norm{u_{j}Pu_{j}^{-1}-P}_{HS}<\eps\norm{P}_{HS}$;
\item[(b)] $\mid\tau(u_{j}) - \left<u_{j}P,P\right>_{HS}/\left<P,P\right>_{HS}\mid<\eps$.
\ei
\ei
Note that in 1), $G=U(M)$ acts on $M$ by conjugation and \S 3, Note applies. 

Bekka's theorem, referred to as Theorem B, is as follows. \\
(Theorem B) Let $\pi$ be a unitary representation of a locally compact group $G$ on a Hilbert space $\mfH$.  Then the following are equivalent:
\bi
\item[1)] $B(\mfH)$ is $G$-amenable (under the conjugation action of \S 5, Example 4).
\item[2)] Given $\eps>0$ and $C\in \mathcal{C}(G)$, there exists a finite rank projection $P\in B(\mfH)$ such that 
$\norm{\pi(x)P\pi(x)^{-1} - P}_{1}<\eps\norm{P}_{1}$  for all $x\in C$ (where $\norm{.}_{1}$ is the norm on the space $TC(\mfH)$ of trace class operators on $\mfH$).  
\ei

Comparing these, Theorem A is set in the context of $II_{1}$-factor M with $G$ the unitary group of $M$ treated as discrete, while in Theorem B, $G$ is an arbitary locally compact group with a representation on an arbitrary Hilbert space $\mfH$.  Both are $G$-amenability theorems.  The condition 2)(a) of Theorem A differs from 2) of Theorem B in that one involves the Hilbert-Schmidt norm while the other involves the trace class norm, but it is easy to go from one norm to the other.  There is no version of Theorem A, 2)(b) in Theorem B.  We will show that the two theorems can be unified in terms of $G$-amenability for a semifinite von Neumann algebra.  

We first summarize some results following from the spectral theorem.  Let $\mfH$ be a Hilbert space and $x\in B(\mfH)$.  Write $x=u(x)\left|x\right|$, the polar decomposition of $x$: so (\cite[p.51]{Murphy}) $\left|x\right|=(x^{*}x)^{1/2}$ and $u(x)$ is the partial isometry that sends $\left|x\right|\xi\to x\xi$ for $\xi\in \mfH$ and vanishes on 
$\ov{\left|x\right|(\mfH)}^{\perp}$.  Also, $u(x)^{*}x=\left|x\right|$.  Let $h\geq 0$ in $B(\mfH)$.  Then by the spectral theorem (\cite[p.72]{Murphy}) there exists a spectral measure $E(h)$ on the spectrum $\si(h)$ of $h$ such that $h=\int t\,d\,E(h)$.  For every bounded Borel measurable function $f$ on $\R$, there is defined an operator $f(h)=\int f(x)\,d\,E(h)(x)$ on $\mfH$ that belongs to the von Neumann algebra generated by $h$.  For $a\geq 0$, let $E_{a}(h)=E(h)((a,\infty))$ which is, of course, a projection in $M$.  

Now let $M$ be a semifinite von Neumann algebra, i.e. a direct sum of type I, type $II_{1}$ and type $II_{\infty}$ von Neumann algebras.  Then (\cite[p.317]{Takv1}) $M$ admits a faithful, semifinite, normal trace $\tau$.   A projection $e\in M$ will be said to be {\em of finite, non-zero trace} if $0<\tau(e)<\infty$.  Note that every projection $e$ of non-zero, finite trace is a finite projection in $M$.  Let $\mfn_{\tau}$ be the set of elements 
%this is m^{1/2} in D1 notation, p.100
$x\in M$ such that $\tau(x^{*}x)<\infty$.  Then (\cite[p.100]{D1}, \cite[p.322]{Takv1}) $\mfn_{\tau}$ is a full Hilbert algebra with scalar product given by: $\left<x,y\right>=\tau(y^{*}x)$.  The completion of this pre-Hilbert space is the Hilbert space
$L^{2}(M,\tau)$ on which $M$ acts by extending the left multiplication (\cite[Theorem 2.22]{Takv1}).  This identifies $M$ with a von Neumann subalgebra of $B(L^{2}(M,\tau))$.  Let $\mfm_{\tau}$ be the set of finite sums of elements of the form $xy$ where $x,y\in \mfn_{\tau}$.  This is an ideal in $M$ and $\tau$ extends from $\mfm_{\tau}\cap M_{+}$ to a linear functional, also denoted $\tau$, on $\mfm_{\tau}$.   Further, $\mfm_{\tau}$ is a normed space under the norm $\norm{.}_{1}$ where $\norm{x}_{1}=\tau(\left|x\right|)$.  The completion of this space is a Banach space 
$L^{1}(M,\tau)$ which (\cite[Theorem 2.18]{Takv1}) is the predual of $M$ under the duality map 
$(y,x)\in M\x \mfm_{\tau}\to \tau(yx)$.  Of course, when $M=B(\mfH)$, then $L^{2}(M,\tau), L^{1}(M,\tau)$ are respectively the spaces of Hilbert-Schmidt and trace-class operators.  The elements of $L^{1}(M,\tau), L^{2}(M,\tau)$ can be realized in terms of (usually unbounded) operators on $L^{2}(M,\tau)$ (\cite{Segal,D1,Nelson,Leinert}) and this theory is used in \cite{Connes}.  However, as in \cite{Bekka}, we are able to avoid the use of unbounded operators in the present context.  

For $a\in \mfm_{\tau}\subset L^{1}(M,\tau)$, let $T_{a}\in B(L^{2}(M,\tau))$ be the operator associated with $a$ by left multiplication.  Of course, as an operator on $L^{2}(M,\tau)$, $T_{a}$ is identified with $a\in M$.  

\begin{proposition}  \label{prop:positive}
\bi
\item[(i)] An element $a\in \mfm_{\tau}$ is a positive linear functional on $M$ if and only if $T_{a}\geq 0$.
\item[(ii)] $S_{*}(M)\cap \mfm_{\tau}$ is norm dense in $S_{*}(M)$.
\ei
\end{proposition}
\begin{proof} (i) Suppose that $a\in \mfm_{\tau}$ is a positive linear functional in $M_{*}$.  Then for $x\in \mfn_{\tau}$, $0\leq a(xx^{*})=\tau(axx^{*})=\tau(x^{*}ax)=\left<T_{a}x,x\right>$, and the density of $\mfn_{\tau}$ in 
$L^{2}(M,\tau)$ gives that $T_{a}\geq 0$.  The converse follows by reversing the preceding argument, and using the ultraweak density of $\mfm_{\tau}^{+}$ in $M^{+}$ and the ultraweak continuity of $a$. \\
(ii) Let $b\in S_{*}(M)$.  Since $b$ is an ultraweakly continuous state, there exists a sequence $\{\xi_{i}\}$ in $L^{2}(M,\tau)$ such that $b=\sum_{i=1}^{\infty}\om_{\xi_{i}}$, where $\sum_{i=1}^{\infty} \norm{\xi_{i}}^{2}=1$ and 
$\om_{\xi_{i}}(m)=\left<m\xi_{i},\xi_{i}\right>$ for all $m\in M$.  Since $\sum_{i=N}^{\infty} \norm{\xi_{i}}^{2}\to 0$ as $N\to \infty$, we can approximate $b$ arbitrarily closely in norm by a state that is a finite sum $r$ of $\om_{\xi_{i}}$'s.  Since $\mfn_{\tau}$ is dense in $L^{2}(M,\tau)$, we can assume that the $\xi_{i}$'s belong to $\mfn_{\tau}$.   But then
$\om_{\xi_{i}}$ is identified with $\xi_{i}\xi_{i}^{*}\in \mfm_{\tau}\subset L^{1}(M,\tau)$, and $r\in S_{*}(M)\cap \mfm_{\tau}$.
\end{proof}  
 
Next suppose that the action $\al$ of $G$ on $M$ is {\em invariant} in the sense that 
for all $m\in M^{+}$, we have $\tau(\al_{x}(m))=\tau(m)$.  Here are two important examples of invariant actions.  If $\al$ is the action on $M$ coming from a representation $\pi$ of $G$ in $M$, then trivially, $\al$ is invariant.  Also, if $M=L^{\infty}(G)$, $\al_{x}\phi=\phi x^{-1}$ (\S 3, (c)) and $\tau$ is a left Haar measure on $G$, then $\al$ is invariant. 

Let $\al$ be an invariant action on $M$.  Then 
$\al_{x}(\mfm_{\tau})=\mfm_{\tau}$ and $\al_{x}(\mfn_{\tau})=\mfn_{\tau}$.  Now let $m'\in \mfm_{\tau}$.  Then we can regard 
$m'$ either as an element of $M_{*}$ or as an element of $M$.  The invariance of $\tau$ gives us the following equality that relates the $G$-actions on $M_{*}$ and $M$: 
\begin{equation}         \label{eq:starx}
(\al_{x^{-1}})_{*}(m')=\al_{x}(m').
\end{equation}
To prove this, for $m\in M$, $\left<(\al_{x^{-1}})_{*}(m'),m\right>=\left<m',\al_{x^{-1}}(m)\right>
=\tau(\al_{x^{-1}}(m^{*})m')=\tau(\al_{x^{-1}}(m^{*}\al_{x}(m')))=\tau(m^{*}\al_{x}(m'))=\left<\al_{x}(m'),m\right>$.
The next proposition is a slight improvement of part of Proposition~\ref{prop:amen}. 

\begin{proposition}        \label{prop:amen2}
Let $x\to \al_{x}$ be an action of $G$ on the semifinite von Neumann algebra $M$.  Then 
$M$ is $G$-amenable if and only if, for each $\eta>0$ and $C\in \mathcal{C}(G)$, there exists  $m'\in \mfm_{\tau}$ such that
$m'\geq 0, \tau(m)=1$ and 
\begin{equation}   \label{eq:21reiter}
\norm{\al_{x}(m') - m'}_{1}<\eta
\end{equation}
for all $x\in C$.
\end{proposition}    
\begin{proof}
By (ii) of Proposition~\ref{prop:positive}, we can take the $a$ of (\ref{eq:reiter1}) to be an element $m'$ in $\mfm_{\tau}$.  Applying (\ref{eq:starx}) gives (\ref{eq:21reiter}).  Then $m'\in S_{*}(M)$ and so is a state on $M$.  Hence $m'\geq 0$ and $1=m'(1)=\tau(m')$.
\end{proof}

We now prove our F\o lner condition for $M$ - the proof is along the same lines as Bekka's Theorem B. 

\begin{theorem}   \label{th:folner}
Let $M$ be a semifinite von Neumann algebra with faithful, semifinite, normal trace $\tau$, and $\al$ be an invariant action of $G$ on $M$.  Then $M$ is $G$-amenable if and only if it satisfies the following F\o lner condition:  given $C\in \mathcal{C}(G)$ and $\eps>0$, there exists a (finite) projection $P$ in $M$ of finite, non-zero trace such that for all $x\in C$,
\begin{equation}
\norm{\al_{x}(P) - P}_{1}<\eps\tau(P).  \label{eq:folner}
\end{equation}
\end{theorem}
\begin{proof}
(\ref{eq:starx}) and
(\ref{eq:folner}) imply the $G$-amenability of $M$ by Proposition~\ref{prop:amen}.  For the converse,  suppose that $M$ is $G$-amenable.  By the Powers-St\o rmer inequality (\cite[Proposition 1.2.1]{Connes}), for $h,k\geq 0$ in $\mfn_{\tau}$, we have 
\begin{equation}  \label{eq:ps}
\norm{h - k}_{2}^{2}\leq \norm{h^{2} - k^{2}}_{1}\leq \norm{h-k}_{2}(\norm{h}_{2} + \norm{k}_{2}).       
\end{equation}
Let $C\in \mathcal{C}(G), \la(C)>0$, $\eps,\de>0$ and $\eta=(\de\eps^{2}/[8\la(C)])^{2}$.  
From Proposition~\ref{prop:amen2}, there exists $m'\in \mfm_{\tau}$, where  $m'\geq 0, \tau(m')=1$, such that 
(\ref{eq:21reiter}) holds.  By Proposition~\ref{prop:positive}, (i), $m'$ is a positive element of $B(L^{2}(M,\tau))$.  Let
$m=(m')^{1/2}\in \mfn_{\tau}$.  Then $\norm{m}_{2}=\tau(m^{2})^{1/2}=1$, and from (\ref{eq:ps}), for all $x\in C$,
\begin{equation}  \label{eq:xmx-1}
 \norm{\al_{x}(m)-m}_{2}<\eta^{1/2}=\de\eps^{2}/[8\la(C)].  
\end{equation}      
Recall that $E(m)$ is the spectral measure of $m$.  
%Note also that for any $h\geq 0$ in $\mfn_{\tau}$, the map $a\to \norm{E_{a}(h)}_{2}^{2}$ ($a\geq 0$) is measurable since %it is increasing.

From the proof of  \cite[Lemma 1.2.6]{Connes}, for $h, k\in \mfn_{\tau}^{+}$ and $a>0$ in $\R$,
\begin{equation}    \label{eq:eahk}
\int_{(0,\infty)}\norm{E_{a}(h^{2}) - E_{a}(k^{2})}_{2}^{2}\,da=\norm{h^{2}-k^{2}}_{1}\leq
\norm{h-k}_{2}\norm{h+k}_{2}.
\end{equation}
Since $\chi_{(a^{1/2},\infty)}(t)=\chi_{(a,\infty)}(t^{2})$, the spectral theorem gives that 
\begin{equation}  \label{eq:21}
E_{a^{1/2}}(h)=E_{a}(h^{2}), \hspace{.2in}   E_{a^{1/2}}(k)=E_{a}(k^{2}).
\end{equation}
%It follows from (\ref{eq:eahk}) and (\ref{eq:21}) that 
%\begin{equation}   \label{eq:coea}
%\int_{(0,\infty)}\norm{E_{a^{1/2}}(h) - E_{a^{1/2}}(k)}_{2}^{2}\,da=\norm{h^{2}-k^{2}}_{1}
%\leq \norm{h-k}_{2}\norm{h+k}_{2}.        
%\end{equation}

We now apply (\ref{eq:eahk}) with $h=\al_{x}(m)$ and $k=m$.   Noting that $\norm{h+k}_{2}\leq \norm{h}_{2}
+\norm{k}_{2}=2$, and using (\ref{eq:xmx-1}), (\ref{eq:eahk}) and (\ref{eq:21}), we obtain
\begin{equation}         \label{eq:deeps8}
\int_{(0,\infty)} \norm{E_{a^{1/2}}(\al_{x}(m)) - E_{a^{1/2}}(m)}_{2}^{2}\,da < \de\eps^{2}/[4\la(C)].
\end{equation}
From the proof of the spectral theorem, $E_{a^{1/2}}(\al_{x}(m))=\al_{x}(E_{a^{1/2}}(m))$.  
Also for $b>0$, 
\begin{equation}  \label{eq:finite}  
E_{b}(m)=\int_{b}^{\infty}\,dE_{\la}\leq b^{-2}\int_{b}^{\infty}\la^{2}\,dE_{\la}\leq b^{-2}m^{2}
\end{equation}
and so 
\begin{equation}  \label{eq:ebmfinite}
\tau(E_{b}(m))\leq b^{-2}\tau(m^{2})<\infty.   
\end{equation}
Using (\ref{eq:ps}) and the continuity of 
$x\to \bt_{x}(c)$ on $M_{*}=L^{1}(M,\tau)$ for fixed $c\in M_{*}$, it follows that the map
$x\to \al_{x}(E_{b}(m))$ is norm continuous from $G$ into $L^{2}(M,\tau)$.

From (\ref{eq:eahk}) and (\ref{eq:21}) (with $h=m,k=0$), we have
\begin{equation}   \label{eq:square}
\int_{(0,\infty)}\norm{E_{a^{1/2}}(m)}_{2}^{2}\,da=\norm{m^{2}}_{1}=1.
\end{equation}
From (\ref{eq:square}) and (\ref{eq:deeps8}), we get
\begin{multline}    \label{eq:geq}
\int_{(0,\infty)}[(\de\eps^{2}/4)\norm{E_{a^{1/2}}(m)}_{2}^{2}-
(\int_{C}\norm{\al_{x}(E_{a^{1/2}}(m)) - E_{a^{1/2}}(m)}_{2}^{2}\,d\la(x))]\,da \\
>\de\eps^{2}/4 -\la(C)[\de\eps^{2}/(4\la(C))]=0.
\end{multline}
So there exists $a>0$ such that 
\begin{equation}   \label{eq:deepsa}
(\de\eps^{2}/4)\norm{E_{a^{1/2}}(m)}_{2}^{2}>\int_{C}\norm{\al_{x}(E_{a^{1/2}}(m)) - E_{a^{1/2}}(m)}_{2}^{2}\,d\la(x).  
\end{equation} 
Let $N$ be the (relatively open) set of $x$'s in $C$ such that 
\begin{equation}   \label{eq:neps}
\norm{\al_{x}(E_{a^{1/2}}(m)) - E_{a^{1/2}}(m)}_{2}^{2}<(\eps^{2}/4)\norm{E_{a^{1/2}}(m)}_{2}^{2}.
\end{equation}
Then by (\ref{eq:deepsa}),
\begin{multline*}
(\eps^{2}/4)\la(C\setminus N)\norm{E_{a^{1/2}}(m)}_{2}^{2}\leq \int_{C}\norm{\al_{x}(E_{a^{1/2}}(m)) - E_{a^{1/2}}(m)}_{2}^{2}\,d\la(x) \\
< (\de\eps^{2}/4)\norm{E_{a^{1/2}}(m)}_{2}^{2},  
\end{multline*}
and it follows that $\la(C\setminus N)<\de$. Let $Q=E_{a^{1/2}}(m)$.  By (\ref{eq:ebmfinite}) and (\ref{eq:deepsa}), $0<\tau(Q)<\infty$.   Also, 
$\norm{Q}_{2}^{2}=\tau(Q^{*}Q)=\tau(Q)$.  Let $x\in N$.  Using (\ref{eq:ps}), (\ref{eq:neps}) and the invariance of $\al$,  for $x\in N$, 
\begin{multline}
\norm{\al_{x}(Q)-Q}_{1}=\norm{(\al_{x}(Q))^{2} - Q^{2}}_{1}\leq
\norm{\al_{x}(Q)-Q}_{2}[2(\tau(Q))^{1/2}]\\
<(\eps/2)(\tau(Q))^{1/2}[2(\tau(Q))^{1/2}]=\eps\tau(Q).  
\end{multline}

So we have proved that given $C\in \mathcal{C}(G)$ and $\eps,\de>0$, there exists a projection $Q\in M$ of finite, non-zero trace and an open subset $N$ of $C$ with $\la(C\setminus N)<\de$ such that
$\norm{\al_{x}(Q) - Q}_{1}<\eps\tau(Q)$ for all $x\in N$. We now want to use the argument of \cite[Theorem 4.10]{Paterson} to establish (\ref{eq:folner}).  To do this, as in that theorem, we take $C\in \mathcal{C}(G)$ ($\la(C)>0$) and $\eps>0$.  Let $D=C\cup C^{2}$ and $\de=\la(C)/2$.  Find a projection $P$ in $M$, $0<\tau(P)<\infty$ such that there is an open subset $N$ of $D$ for which $\la(D\setminus N)<\de$ and $\norm{\al_{x}(P) - P}_{1}<(\eps/2)\tau(P)$ for all $x\in N$.  Then if $x,y\in N$, using the invariance of $\al$, we get $\norm{\al_{xy^{-1}}(P) - P}_{1}\leq \norm{\al_{xy^{-1}}(P) - \al_{x}(P)}_{1}
+\norm{\al_{x}(P)-P}_{1}= \norm{\al_{y}(P)-P}_{1}+\norm{\al_{x}(P)-P}_{1}<\eps\tau(P)$.  It is proved in 
\cite[Theorem 4.10]{Paterson} that $C\subset NN^{-1}$ and (\ref{eq:folner}) is proved.
\end{proof}

We note that in the case $M=B(\mfH)$ with the standard trace $\tau$ and where the action on $M$ comes from a representation $\pi$ of $G$ in $M$, then the projection $P$ satisfies $\tau(P)<\infty$ if and only if it is finite-dimensional 
(\cite[Part 1, Chapter 6, Theorem 5]{D1}), and so the above theorem in that case reduces to Theorem B.  The classical F\o lner condition is obtained from Theorem~\ref{th:folner} from the case (c) of \S 3 where $M=L^{\infty}(G)$, $\al_{x}\phi=\phi x^{-1}$ and $\tau$ is a left Haar measure $\la$ on $G$.  In that case, $G$-amenability is the same as the amenability of $G$, and the theorem gives that this is equivalent to the existence of a measurable set $E$ in $G$ of finite positive measure such that 
$\la(xE\bigtriangleup E)/\la(E)<\eps$.  The classical F\o lner condition follows using the inner regularity of $\la$.   We now formulate a $\cs$ version of condition 2b) of Theorem A above. 

\begin{theorem}          \label{th:traceest}
Let $M$ be a semifinite von Neumann algebra with faithful, semifinite, normal trace $\tau$.  Let $G$ be a locally compact group, $\pi$ a unitary representation of $G$ in $M$ and $B\subset M$ be the $\cs$ generated by $\pi(L^{1}(G)\cup G)$. 
Then there exists a net of projections $\{e_{\de}\}$ in $M$ such that 
$0<\tau(e_{\de})<\infty$ for all $\de$ and a tracial state $\si$ on $B$ such that 
\begin{equation}  \label{eq:eden}
g_{\de}\to \si
\end{equation}
weak$^{*}$ in $B^{*}$, where $g_{\de}$ is the restriction of $p_{\de}$ to $B$ and
$p_{\de}\in S(M)$ given by: $p_{\de}(m)=\tau(me_{\de})/\tau(e_{\de})$.
\end{theorem}
\begin{proof}
By (\ref{eq:folner}), there exists a net $\{e_{\de}\}$ of projections in $M$ of finite, non-zero trace such that 
$\norm{\pi(x)e_{\de}' - e_{\de}'\pi(x)}_{1}\to 0$ 
uniformly on compacta, where $e_{\de}'=e_{\de}/\tau(e_{\de})$.  As in the proof of Proposition~\ref{prop:hyperamen}, it follows that $\norm{\pi(f)e_{\de}' - e_{\de}'\pi(f)}_{1}\to 0$ for all $f\in C_{c}(G)$.  We can suppose that $g_{\de}\to \si$ weak$^{*}$ in $B^{*}$.  Then $\si\geq 0$, and since $\si(1)=1$, it follows that $\si$ is a state on $B$.  Also, since for each $m\in M$, $\left|p_{\de}(\pi(f)m - m\pi(f))\right|\leq \norm{\pi(f)e_{\de}' - e_{\de}'\pi(f)}_{1}\norm{m}$, we get 
that $\pi(f)\si=\si\pi(f)$ for all $f\in L^{1}(G)$.   Similarly, $\pi(x)\si =\si\pi(x)$ for all $x\in G$.  Since
$\pi(L^{1}(G)\cup G)$ spans a dense subalgebra of $B$, it follows that $\si$ is a tracial state.
\end{proof}
\begin{corollary}   \label{cor:traceest}
Let $M, G, \pi$ be as in the preceding corollary and let $A$ be the $\cs$ generated by $\pi(L^{1}(G))$.  Suppose that $A$ has a unique tracial state $\si_{0}$ (e.g. $A$ could be UHF).  Then there exists a net of projections $\{e_{\de}\}$ in $M$ of finite, non-zero trace and a constant $k\in [0,1]$ such that $\widehat{h_{\de}}\to k\si_{0}$ weak$^{*}$ in $A^{*}$, where $h_{\de}$ is the restriction of $p_{\de}$ to $A$.
\end{corollary}
\begin{corollary}            \label{cor:tauresult}
Let $M$ be as above and $N$ be a factor of type $II_{1}$ contained in $M$ and with (unique) normalized trace $\tau_{N}$.  Suppose that there exists a $U(N)$-hypertrace on $M$ for the factor $N$, i.e. a state $\phi\in M^{*}$ such that $\phi(mn-nm)=0$ for all $m\in M, n\in N$.  Then there exists a net $\{e_{\de}\}$ of projections in $M$ of finite, non-zero trace such that $h_{\de}\to \tau_{N}$ weak$^{*}$ in $N^{*}$, where $h_{\de}$ is the restriction of $p_{\de}$ to $N$.
\end{corollary}
\begin{proof}
Let $G=U(N)$ with the discrete topology and $\pi$ be the inclusion map from $U(N)$ into $U(M)$.  Then 
$B$ of Theorem~\ref{th:traceest} is just $N$ since $U(N)$ spans $N$.  Then (\S 3, Note) $M$ is $G$-amenable.  Now apply Theorem~\ref{th:traceest}, noting that $\si$ has to be $\tau_{N}$ by the uniqueness of the tracial state on $N$.
\end{proof}

%*******************

\end{document}